\documentclass[11pt]{amsart}

\usepackage{amsmath}
\usepackage{amssymb}
\usepackage{graphicx}
\usepackage{enumerate}

\DeclareMathOperator{\curl}{curl}
\DeclareMathOperator{\re}{Re}
\DeclareMathOperator{\Div}{div}

\newcommand{\R}{\mathbb R}
\newcommand{\C}{\mathbb C}
\newcommand{\Z}{\mathbb Z}
\newcommand{\N}{\mathbb N}
\newcommand{\Hb}{\mathcal{H}_\Phi^{\Omega}}
\newcommand{\Hf}{\mathsf{H}_{\Phi}^{\Omega_0}}

\newcommand{\Ab}{\mathbf A}
\newcommand{\Fb}{\mathbf F}
\newcommand{\ab}{\mathbf a}
\newcommand{\dd}{\mathop{}\!{\mathrm d}}

\newcommand{\ii}{\mathrm i}

\newcommand{\Ef}{\mathcal E_\Phi}
\newcommand{\En}{\mathrm E(\Phi)}
\newcommand{\Gf}{\mathcal G_\Phi}
\newcommand{\Gn}{\mathrm G(\Phi)}
\newcommand{\cI}{\mathcal I}
\newcommand{\cL}{\mathcal L}
\newcommand{\cE}{\mathcal E}

\newcommand{\cX}{\mathcal X}

\newcommand{\cG}{\mathcal G}
\newcommand{\rG}{\mathrm G}
\newcommand{\jb}{\mathbf j}
\newcommand{\cH}{\mathcal H}

\newcommand{\Bl}{\color{blue}}

 \usepackage{xcolor}
\definecolor{blue(ryb)}{rgb}{0.01, 0.28, 1.0}
\definecolor{brandeisblue}{rgb}{0.0, 0.44, 1.0}
\definecolor{ceruleanblue}{rgb}{0.16, 0.32, 0.75}
\definecolor{cobalt}{rgb}{0.0, 0.28, 0.67}
\definecolor{coolblack}{rgb}{0.0, 0.18, 0.39}
\definecolor{darkblue}{rgb}{0.0, 0.0, 0.55}
\usepackage[hidelinks]{hyperref}
\hypersetup{
    colorlinks,
    citecolor=darkblue,
    filecolor=black,
    linkcolor=darkblue,
    urlcolor=black
}

\newtheorem{theorem}{Theorem}[section]
\newtheorem{corollary}[theorem]{Corollary}

\newtheorem{proposition}[theorem]{Proposition}
\theoremstyle{definition}

\newtheorem{remark}[theorem]{Remark}

\numberwithin{equation}{section}
\title[Local strong magnetic fields]{Local strong magnetic fields and the Little-Parks effect}

\author[A. Kachmar]{Ayman Kachmar}
\address{School of Science and Engineering, The Chinese University of Hong Kong Shenzhen, Guangdong, 518172, P.R. China.}
\email{akachmar@cuhk.edu.cn}

\author[M. Sundqvist]{Mikael Sundqvist}
\address{Department of Mathematics, Lund University, Lund, Sweden.}
\email{mikael.persson{\_}sundqvist@math.lth.se}



\begin{document}

\begin{abstract}
Starting from the Ginzburg--Landau model in a planar simply connected domain, with a local compactly supported applied magnetic field, we derive  an effective model in the strong field limit, defined on a non-simply connected domain. The effective model features oscillations in the Little-Parks and Aharonov--Bohm spirit. We also discuss a similar question for the lowest eigenvalue of the magnetic Laplacian.
\end{abstract}

\maketitle

%%%%%%%%%%%%%%%%%%%%%%%%%%%%%%%%%%%%%%%%%%%%%%%%%%%%%%%%%%%%%%%%%%%%%%%%%%%%%%%%
\section{Introduction}\label{sec:intro}

\subsection{The Little-Parks effect}
As the applied  magnetic flux varies, superconductors undergo phase transitions between normal and superconducting states. The Little-Parks effect is observed when the transition is non-monotone and exhibits oscillations between superconducting and normal phases \cite{LP}. Within the Ginzburg-Landau theory,
examples of this sort  were provided previously  in several settings:
\begin{enumerate}[a)]
\item Under a constant magnetic field, when imposing a Robin condition \cite{G1, KS} or when restricting the superconducting sample to a thin domain \cite{FS, HK, RS, ShSt};
\item under Aharonov-Bohm magnetic fields  \cite{KP};
\item under certain positive non-homogeneous radial magnetic fields \cite{FS}.
\end{enumerate}

In this paper, we consider a strong local magnetic field and prove that oscillations occur indefinitely in the Little--Parks spirit.

\subsection{The local magnetic field}

Consider two simply connected domains \(\omega,\Omega\subset\R^2\) such that \(\overline{\omega}\subset\Omega\), and assume that their boundaries, \(\partial\omega\) and \(\partial\Omega\),  are simple smooth \(C^1\) curves.
By a local magnetic field we mean a non homogeneous magnetic field of the form $b\mathbf 1_\omega\hat z$, where $\hat z=(0,0,1)^\top$, $b>0$ and $\mathbf 1_\omega$ is the characteristic function of $\omega$,
\begin{equation}\label{eq:def-B}
\mathbf 1_\omega(x)=\begin{cases}1&\text{ if }x\in\omega,\\
0&\text{ if }x\not\in\omega.
\end{cases}
\end{equation}
This is an example of a magnetic step, which has recently received considerable attention in the context of superconductivity and spectral theory \cite{AG, AHK, AKPS, G2}.
Corresponding to the magnetic field $b\mathbf 1_\omega\hat z$ is the magnetic flux
\begin{equation}\label{eq:def-flux}
\Phi:=\frac1{2\pi}\int_{\Omega}b\mathbf 1_\omega\dd x=\frac{b|\omega|}{2\pi}.
\end{equation}
We introduce the vector potential \(\Fb:\overline{\Omega}\to\mathbb R^2\) defined as
\begin{equation}\label{eq:def-F}
\Fb=(-\partial_2\phi,\partial_1\phi),
\end{equation}
where \(\phi\) is the unique solution of
\[ -\Delta\phi=\frac{2\pi}{|\omega|}\mathbf 1_\omega\quad\text{ in } \Omega,\quad \phi=0\text{ on }\partial\Omega,\]
and notice that  \(\Fb\) satisfies
\begin{equation}\label{eq:prop-A}
\curl\Fb=\frac{2\pi}{|\omega|}\mathbf 1_\omega,\quad  {\rm div}\Fb=0\text{ on }\Omega,\quad\text{ and } \nu\cdot\Fb=0\text{ on }\partial\Omega,
\end{equation}
where \(\nu\) is the unit normal to \(\partial\Omega\) pointing inward \(\Omega\). Hence, $\Phi\Fb$ is a vector potential corresponding to the magnetic field $b\mathbf 1_\omega\hat z$, as $\curl(\Phi\Fb)=b\mathbf 1_\omega$ in view of \eqref{eq:def-flux}.\medskip
\subsection{The Ginzburg-Landau model}
Let us suppose that $\Omega$ is the cross section of a cylindrical superconductor subject to the magnetic field  $b\mathbf 1_\omega\hat z$, with corresponding magnetic flux $\Phi$ in \eqref{eq:def-flux}. As the magnetic flux $\Phi$ increases, it is well known that the superconductor undergoes a phase transition between the superconducting and normal states.  The state of superconductivity  is described by a configuration $(\psi,\Ab)$ such that
\[
    \psi:\Omega\to\C,\quad \Ab:\Omega\to\R^2,
\]
with $|\psi|^2$ representing the local density of superconductivity,  $\curl\Ab$ representing the local induced magnetic field, and
\begin{equation}
    \label{eq:def-supercurrent}
    \jb(\psi,\Ab)
    :=
    \re \langle \psi,(-\ii\nabla - \Ab)\psi\rangle_{\C}
\end{equation}
measures the supercurrent. At equilibrium, the physically interesting configurations are the minimizers (or even critical points) of the energy functional
\begin{equation}\label{eq:def-GL}
\Ef(\psi,\Ab)=\int_\Omega\Biggl(|(-\ii\nabla-\Phi \Ab)\psi|^2-\kappa^2|\psi|^2+\frac{\kappa^2}{2}|\psi|^4+\Phi^2|\curl(\Ab-\Fb)|^2\Biggr)\dd x,
\end{equation}
defined on the following function space
\begin{equation}\label{eq:def-space}
    \cH
    =
    \{
        (\psi,\Ab)\in H^1(\Omega;\C)\times H^1(\Omega;\R^2)\colon
        \Div\Ab=0\text{ on }\Omega,~\nu\cdot\Ab=0\text{ on }\partial\Omega
    \}.
\end{equation}
The definition of the functional $\Ef$ involves a positive constant $\kappa$ called the Ginzburg--Landau parameter, which is a characteristic of the superconducting material. Throughout this paper, $\kappa$ will be fixed.

We introduce also the ground state energy \(\En\) of the functional $\Ef$ as
\begin{equation}\label{eq:def-energy}
\En=\inf\{\Ef(\psi,\Ab)\colon (\psi,\Ab)\in\cH\}.
\end{equation}
We say that a configuration $(\psi,\Ab)$ is a minimizer of $\Ef$ if $(\psi,\Ab)\in\cH$ and $\Ef(\psi,\Ab)=\En$.
\subsection{The effective  model}
Our aim is to understand the behavior of $\En$ and the minimizing configurations in the limit of large $\Phi$ (which by \eqref{eq:def-flux} is equivalent to the limit of large $b$). We will show that the strong magnetic field in the interior domain \(\omega\) will essentially force the minimizing \(\psi\) to be zero there. This motivates the introduction of an effective functional, defined on the space
\begin{equation}\label{eq:def-space0}
\cH_0=\{(u,\Ab)\in \cH\colon u=0\text{ on }\omega\},
\end{equation}
as
\begin{multline}\label{eq:def-GL0}
\Gf(u,\Ab)=
\int_{\Omega_0}\Bigl(|(-\ii\nabla-\Phi \Ab)u|^2-\kappa^2|u|^2+\frac{\kappa^2}{2}|u|^4 \Bigr)\dd x\\+\Phi^2\int_{\Omega}|\curl(\Ab-\Fb)|^2\dd x,
\end{multline}
where $\Omega_0:=\Omega\setminus{\overline\omega}$.  Note that $\Gf$ acts on configurations $(u,\Ab)$ where the function $u$ has support in the non-simply connected domain $\Omega_0$. We can and will view it as being defined in \(\Omega_0\) with a Dirichlet boundary condition $u=0$ on $\partial\omega$. In fact,  thanks to the smoothness of $\partial\omega$, a function $u$ belongs to the space
\begin{equation}\label{eq:def-space0*}
\cX(\Omega_0)=\{u\in H^1(\Omega_0;\C)\colon u=0\mbox{ on }\partial\omega\}
\end{equation}
if and only if it has an extension $\widetilde u\in H^1(\Omega;\C)$ and $\widetilde u=0$ on $\omega$.

The functional $\Gf$ turns out to be the relevant approximation of the functional $\Ef$ in the limit of large $\Phi$. We describe this approximation in Theorem~\ref{thm:GL} below, which  involves the  effective ground state energy that we define as
\begin{equation}\label{eq:def-energy0}
\Gn=\inf\{\Gf(u,\Ab)\colon (u,\Ab)\in \cH_0\}.
\end{equation}
Thanks to invariance under gauge transformations, it suffices to study $\Gn$  for $0\leq \Phi<1$. In fact, on $\Omega_0$, we define the  function
\begin{equation}\label{eq:def-U-1}
U_n(x)=\exp\left(\ii n\int_{\ell(x_*,x)}\Fb\cdot\dd \mathrm r\right),
\end{equation}
where $x_*$ is a fixed point in $\Omega_0$ and $\ell(x_*,x)$ is any path in $\Omega_0$ connecting $x_*$ and $x$. The definition of $U_n$ is independent of the choice of the path and the relevance of $U_n$ is that  it is $C^1$ smooth and satisfies (see \cite[p. 21]{FH-b})
\begin{equation}\label{eq:prop-U1}
\nabla U_n=\ii  n U_n \Fb \text{ on }\Omega_0.
\end{equation}
Consequently, we have the following identity
\begin{equation}\label{eq:prop-U1*}
U_n^{-1}(-\ii\nabla-\Phi\Fb)U_n=-\ii\nabla-(\Phi-n)\Fb,
\end{equation}
which allows us to shift $\Phi$  by an integer  without changing the energy $\Gn$.
\begin{theorem}\label{thm:GL}
As $\Phi\to+\infty$, the ground state energies $\En$ and $\Gn$ introduced in \eqref{eq:def-energy} and \eqref{eq:def-energy0} satisfy
\[\En= \Gn+o(1).\]
Moreover, it holds the following.
\begin{enumerate}[\rm i)]
\item The function $\Phi\mapsto \Gn$ is periodic with period $1$.
\item If $0\leq \Phi_0<1$ and $(\psi_n,\Ab_n)_{n\geq 1}$ is a sequence such that, every $(\psi_n,\Ab_n)$ is a minimizing configuration of $\Ef$ for $\Phi=\Phi_n:=\Phi_0+n$, then there exists a minimizer $(u_*,\Ab_*)\in \cH_0$ of $\cG_{\Phi_0}$ and a subsequence $(n_k)$ such that
\[\begin{gathered}
|\psi_{n_k}|\to |u_*|\text{ in }H^1(\Omega; \R),\\
U_{n_k}^{-1}\psi_{n_k}\to u_*\text{ in }H^1(\Omega_0;\C),\\ \jb(\psi_{n_k},\Ab_{n_k})\to\jb_*\text{ in }L^2(\Omega;\R^2),
\end{gathered}\]
where  $\jb_*$ is defined as
\[\quad \jb_*= {\re} \langle u_*,(-\ii\nabla-{\Phi_0}\Ab_*)u_*\rangle_\C.
\]
\item If $\Phi\in \N$, then any minimizer $(u,\Ab)\in\cH_0$ of $\Gf$ satisfies
\[ |u|>0 \text {\ { in \(\Omega _ 0\)}},\quad \Ab=\Fb,\quad \jb(u,\Ab)=0.\]
\end{enumerate}
\end{theorem}

We will prove Theorem~\ref{thm:GL} in Section~\ref{sec:funcconv}.

\subsection{Oscillations for the effective model}
Consider the magnetic Laplacian in the non-simply connected domain $\Omega_0=\Omega\setminus\overline{\omega}$,
\begin{equation}\label{eq:def-op-fl}
\Hf=(-\ii\nabla-\Phi\Fb)^2\,,
\end{equation}
with Neumann boundary condition
$\nu\cdot \nabla u=0$ on $\partial\Omega$, and Dirichlet boundary condition $u=0$ on $\partial\omega$.
 Since we do not impose Neumann boundary condition on the inner boundary \(\partial\omega\), \(\Hf\) differs from the Neumann realization  considered in~\cite{HHOO}, but it can still be studied using the same methods as in~\cite{HHOO}.
The operator $\Hf$ has compact resolvent, its spectrum is purely discrete, and its lowest eigenvalue is
\begin{equation}\label{eq:def-ev-Hf}
\lambda_1^{\Omega_0}(\Phi)=\inf_{\substack{u\in \cX(\Omega_0)\\u\not=0}}\frac{\|(-\ii\nabla-\Phi\Fb)u\|^2_{L^2(\Omega_0)}}{\|u\|^2_{L^2(\Omega_0)}},
\end{equation}
where $\cX(\Omega_0)$ is the space introduced in \eqref{eq:def-space0*}.

With the help of $\lambda_1^{\Omega_0}(\Phi)$, we may characterize whether a minimizer of $\Gf$ is identically zero or not.  Actually, it is straightforward to verify that
\begin{equation}\label{eq:criterion}
\lambda_1^{\Omega_0}(\Phi)<\kappa^2\Longrightarrow\Gn<0,
\end{equation}
by using $(tv,0)$ as a trial state, with $v$ a ground state of $\lambda_1^{\Omega_0}(\Phi)$ and $t>0$ sufficiently small.
Since $\Gn\leq \Gf(0)=0$, then denoting by $u_\Phi$ a minimizer of $\Gf$, we have that $u_\Phi$  is not identically zero whenever $\lambda_1^{\Omega_0}(\Phi)<\kappa^2$.

The transition from $u_\Phi=0$ to $u_\Phi\not=0$ is then related to the monotonicity of the eigenvalue $\lambda_1^{\Omega_0}(\Phi)$, which we would like to explore next.
The diamagnetic inequality yields that
\begin{equation}\label{eq:dia-mag}
\lambda_1^{\Omega_0}(\Phi) \geq \lambda_1^{\Omega_0}(0)>0,
\end{equation}
and thanks to \eqref{eq:prop-U1*}, the following holds.
\begin{proposition}[{\cite[Theorem~1.1]{HHOO}}]\label{prop:Hf}
The function $\Phi\mapsto\lambda_1^{\Omega_0}(\Phi)$ is continuous, periodic with period $1$,   non-constant, and its minimum is attained at $\Phi=0$  while its maximum is attained at $\Phi=\frac12$. Moreover,
\begin{equation}\label{eq:ev-flux}
\begin{gathered}
\lambda_1^{\Omega_0}(\Phi+\mbox{$\frac12$})=\lambda_1^{\Omega_0}(\Phi-\mbox{$\frac12$})\,,\\
\lambda_1^{\Omega_0}(\Phi)>\lambda_1^{\Omega_0}(0)\quad\Bigl(\Phi\not\in \N\Bigr)\,,
\end{gathered}
\end{equation}
and
\begin{equation}\label{eq:ev-max}
\lambda_1^{\Omega_0}(\Phi)<\lambda_1^{\Omega_0}(\mbox{$\frac12$})\quad\Bigl(\Phi\not\in \mbox{$\frac12$}+\N\Bigr).
\end{equation}
\end{proposition}

Consequently, we get that minimizers of $u_\Phi$ oscillate between $u_\Phi=0$ (normal state) and $u_\Phi\not=0$ (superconducting state) as  $\Phi$ varies.
\begin{corollary}\label{cor:oscillations1}~
\begin{enumerate}[\rm A)]
\item There exists a positive constant $c_*\in(0,1]$ such that the following is true. Suppose that $0<\kappa<c_*\sqrt{\lambda_1^{\Omega_0}(1/2)}$,  and consider the two sequences $(\Phi_n:=n-\frac12)$ and $(\Phi_n':=n)$. Then, for all $n\in \N$, we have $0<\Phi_n<\Phi_n'<\Phi_{n+1}$ and
\begin{enumerate}[\rm i)]
\item $u_{\Phi_n}=0$ for any critical point $(u_{\Phi_n},\Ab_{\Phi_n})$  of $\cG_{\Phi_n}$;
\item $u_{\Phi_n'}\not=0$ for any minimizer $(u_{\Phi_n'},\Ab_{\Phi_n'})$  of $\cG_{\Phi_n'}$.
\end{enumerate}
\item If $\kappa>\sqrt{\lambda_1^{\Omega_0}(1/2)}$, then any minimizer $(u_\Phi,\Ab_\Phi)$ of $\Gf$
satisfies $u_\Phi\not=0$.
\end{enumerate}
\end{corollary}

\subsection{The magnetic Laplacian with a local magnetic field}
In order to  investigate the occurrence of oscillations for the initial model in \eqref{eq:def-GL}, one has to study
\begin{equation}\label{eq:ev-H-B}
\lambda_1^\Omega(\Phi):=\inf_{\substack{u\in H^1(\Omega;\C)\\u\not=0}}\frac{\|(-\ii\nabla-\Phi\Fb)u\|^2_{L^2(\Omega)}}{\|u\|^2_{L^2(\Omega)}},
\end{equation}
 the lowest eigenvalue of the magnetic Laplacian  in $L^2(\Omega)$,
\begin{equation}\label{eq:def-op-Om}
    \Hb= (-\ii\nabla - \Phi\Fb)^2,
\end{equation}
with Neumann boundary conditions $\nu\cdot \nabla u=0$ on \(\partial\Omega\). Subject to the local magnetic field $\Phi\curl\Fb=\frac{2\pi}{|\omega|}\Phi\mathbf 1_\omega=b\mathbf 1_\omega$, this eigenvalue problem is also related to the existence of discrete spectrum for locally deformed leaky wires \cite{BE}.

Guided by Theorem~\ref{thm:GL}, we anticipate that the eigenvalue $\lambda_1^{\Omega_0}(\Phi)$, introduced in \eqref{eq:def-ev-Hf}, approximates the eigenvalue
$\lambda_1^\Omega(\Phi)$ in the limit of large $\Phi$. This can likely be obtained from \cite{HN} in the framework of norm resolvent convergence, but we give here another proof which is closer to the one of Theorem~\ref{thm:GL}.

\begin{theorem}\label{thm:ev}
As $\Phi\to+\infty$, the eigenvalues $\lambda_1^{\Omega}(\Phi)$ and $\lambda_1^{\Omega_0}(\Phi)$ introduced in \eqref{eq:ev-H-B} and \eqref{eq:def-ev-Hf} satisfy
\[
    \lambda_1^{\Omega}(\Phi) = \lambda_1^{\Omega_0}(\Phi) + o(1).
\]
\end{theorem}
Notice that, by comparing with the Dirichlet problem in $\Omega_0$, we have the non-asymptotic bound,
\begin{equation}\label{eq:comparison}
\lambda_1^{\Omega}(\Phi)\leq \lambda_1^{\Omega_0}(\Phi),\quad \forall\, \Phi>0.
\end{equation}
To prove Theorem~\ref{thm:ev}, we will establish the matching asymptotic lower bound,
\begin{equation}\label{eq:comparison*}
\lambda_1^{\Omega}(\Phi)\geq \lambda_1^{\Omega_0}(\Phi)+o(1).
\end{equation}
This will be done in Section~\ref{sec:eigenvalueapproximation}. As a direct consequence of Theorem~\ref{thm:ev}, we get that the minimizers of the functional $\Ef$ in \eqref{eq:def-GL} undergo indefinite oscillations between the normal and superconducting phases.
\begin{corollary}\label{cor:oscillations2}
Suppose that $0<\kappa<c_*\sqrt{\lambda_1^{\Omega_0}(1/2)}$, where $c_*\in(0,1)$ is a constant independent of $\kappa$ and $\Phi$. There exist two sequences $(\Phi_n)$ and $(\Phi_n')$ such that, for all $n\in \N$, we have $0<\Phi_n<\Phi_n'<\Phi_{n+1}$ and
\begin{enumerate}[\rm i)]
\item $\psi_{\Phi_n}=0$ for any critical point $(\psi_{\Phi_n},\Ab_{\Phi_n})$ of $\cE_{\Phi_n}$;
\item $\psi_{\Phi_n'}\not=0$ for any minimizer $(\psi_{\Phi_n'},\Ab_{\Phi_n'})$ of $\cE_{\Phi_n'}$.
\end{enumerate}
Moreover, if $\kappa>\sqrt{\lambda_1^{\Omega_0}(1/2)}$, then for all $\Phi>0$, any minimizer $(\psi_\Phi,\Ab_\Phi)$ of $\Ef$ satisfies $\psi_\Phi\not=0$.
\end{corollary}
To construct the sequences in the above corollary, we note that a minimizer of $\Ef$ is not identically zero whenever $\lambda_1^\Omega(\Phi)<\kappa^2$, and we use Proposition~\ref{prop:Hf} and Theorem~\ref{thm:ev} to construct $(\Phi_n)$ and $(\Phi_n')$ such that $\lambda_1^\Omega(\Phi_n)>\kappa^2$ and $\lambda_1^{\Omega}(\Phi_n')<\kappa^2$.

\section{Convergence to the effective functional}\label{sec:funcconv}

In this section we prove Theorem~\ref{thm:GL} and Corollary~\ref{cor:oscillations1}.
In the various calculations, we will encounter, in addition to the flux $\Phi$, the integer and fractional parts
\begin{equation}\label{eq:def-flux*}
\lfloor\Phi\rfloor=\sup\{m\in\Z\colon m\leq \Phi\},\quad \{\Phi\}=\Phi-\lfloor\Phi\rfloor.
\end{equation}
We start by proving that the effective ground state energy is periodic with respect to $\Phi$.  For any $\Phi>0$ and $n\geq 0$, we define the transformation
\[
    \cL_{\Phi,n}(u,\Ab)
    =
    \left(
    U_nu, \frac{\Phi}{\Phi+n}\Ab+\frac{n}{\Phi+n}\Fb
    \right).
\]

\begin{proposition}\label{prop:GL-eff}
For all $\Phi>0$, we have
\[
    \rG(\Phi+1)=\Gn,
\]
where  $\rG(\cdot)$ is introduced in \eqref{eq:def-energy0}. Moreover, with $n=\lfloor \Phi\rfloor$ and $\Phi_0=\{\Phi\}$, we have
\begin{multline*}
    (u_*,\Ab_*)\text{ is a minimizer of }\cG_{\Phi_0}
    \iff \\
    (u,\Ab)
    =
    \cL_{\Phi_0,n}(u_*,\Ab_*)\text{ is a minimizer of }\Gf.
\end{multline*}
\end{proposition}

\begin{proof}
Let $(u,\Ab)\in \cH_0$, where $\cH_0$ is the space introduced in \eqref{eq:def-space0}. With $\tilde u=U_1u$ and $\tilde\Ab=\frac{\Phi}{\Phi+1}\Ab+\frac{1}{\Phi+1}\Fb$,  the configuration $\bigl(\tilde u,\tilde\Ab)\in \cH_0$, and thanks to \eqref{eq:prop-U1*}, we have
\[\cG_{\Phi+1}(\tilde u,\tilde\Ab)=\Gf(u,\Ab). \]
Minimizing over $(u,\Ab)\in\cH_0$, we get  $\rG\bigl(\Phi+1\bigr)=\Gn$, and  if $(u,\Ab)$ is a minimizer of $\Gf$, then $(\tilde u,\tilde\Ab)$ is a minimizer of $\cG_{\Phi+1}$. By iteration, we cover the case where $\Phi=\Phi_0+n$ and $n\in \N$.
\end{proof}
\begin{remark}\label{rem:trivial}
Suppose that $\Phi\in \N$. Then, by Proposition~\ref{prop:GL-eff},  the set of minimizers of $\Gf$ is $\{(cU_\Phi u_*,\Fb)\colon c\in\C\text{ with }|c|=1\}$,
where $u_*$ is the unique positive solution of
\[
    -\Delta u_*=\kappa^2(1-u^2_*)u_*\text{ on }\Omega_0,
    \quad
    u_*=0\text{ on }\partial\omega,
    \quad
    \nu\cdot\nabla u_*=0\text{ on }\partial\Omega.
\]
Consequently, for a minimizer $(u_\Phi,\Fb)$ of $\Gf$, we have
\[
    |u_\Phi|>0,
    \quad
    \jb(u_\Phi,\Fb)
    =
    \re \langle u_\Phi,(-\ii\nabla -\Phi\Fb)u_\Phi\rangle_{\C}
    = 0.
\]
In fact,  $u_\Phi=cU_\Phi u_*$, with $|c|=1$ and $u_*$ real-valued, and by \eqref{eq:def-U-1},
\[\langle u_\Phi,(-\ii\nabla -\Phi\Fb)u_\Phi\rangle_{\C}=\langle u_*,(-\ii\nabla -\Phi \Fb +\Phi \Fb)u_*\rangle_{\C}=
\ii u_*\nabla u_*.\]
\end{remark}

We  next give a non-asymptotic upper bound on the ground state energy $\En$.

\begin{proposition}\label{prop:ub}
Let $\En$ and $\Gn$ be the ground state energies introduced in \eqref{eq:def-energy} and \eqref{eq:def-energy0} respectively. Then, for all $\Phi\geq 0$, we have
\[ \En\leq \Gn\leq 0.\]
\end{proposition}
\begin{proof}
Firstly,  $\Gn\leq \Gf(0,\Fb)=0$. Secondly, letting $(u,\Ab)\in\cH_0$ be a minimizer of $\Gn$,  we observe that $(u,\Ab)\in\cH$ and  $\En\leq \Ef(u,\Ab)=\Gf(u,\Ab)=\Gn$.
\end{proof}
To establish an asymptotic lower bound on $\En$,  we collect below some known  estimates on the minimizing configurations  of
the functional $\Ef$ (see \cite[Prop. 10.3.1 \& Lem. 10.3.2]{FH-b}), where $\|\cdot\|_p$ denotes the standard norm on $L^p(\Omega)$.
\begin{proposition}\label{prop:est-min}
There exists a positive constant $C$ such that, if $\Phi>0$ and $(\psi_\Phi,\Ab_\Phi)\in\cH$ is a minimizer  of $\Ef$,  then
\[\begin{gathered}
\|\psi\|_\infty\leq 1,\quad \|(-\ii\nabla-\Phi\Ab)\psi\|_2\leq \kappa\|\psi\|_2,\quad \|\curl(\Ab-\Fb)\|_2\leq \frac{\kappa}{\Phi}\|\psi\|_2,
\end{gathered}\]
and
\begin{equation}\label{eq:curl-div}
\|\Ab-\Fb\|_{H^2(\Omega)}\leq \frac{C\kappa}{\Phi}\|\psi\|_2.
\end{equation}
\end{proposition}
Notice that \eqref{eq:curl-div} results from the curl-div inequality
\[ \|\ab\|_{H^2(\Omega)}\leq C_\Omega\|\curl\ab\|_{H^1(\Omega)}\]
valid for any vector field satisfying
\[\ab\in H^1(\Omega;\R^2),\quad \curl\ab\in H^1(\Omega;\R),\quad \mathrm{div\,}\ab=0\text{ on }\Omega,\quad\nu\cdot\ab=0\text{ on }\partial\Omega.\]
The vector field $\ab=\Ab-\Fb$ satisfies the aforementioned conditions  and (see \cite[p.~143]{FH-b})
\begin{equation}\label{eq:curl-div*}
\nabla^\bot\curl(\Ab-\Fb)=-\Phi^{-1}\jb(\psi,\Ab)\text{ on }\Omega,\quad \curl(\Ab-\Fb)=0\text{ on }\partial\Omega,
\end{equation}
where $\nabla^\bot=(-\partial_{x_2},\partial_{x_1})$ and $\jb(\psi,\Ab)$ is the supercurrent introduced in \eqref{eq:def-supercurrent}. We get then \eqref{eq:curl-div} with $C=2C_\Omega$.\medskip

Starting with any sequence of minimizing configurations of $\Ef$, we can  extract, by  a standard compactness argument, limits of the magnetic potential and the modulus of the order parameter.
\begin{proposition}\label{prop:conv-density}
Let  $(\psi_\Phi,\Ab_\Phi)_{\Phi>0}\subset\cH$ be a family of configurations such that $(\psi_\Phi,\Ab_\Phi)$ is a minimizer of $\Ef$ for every $\Phi>0$. Then, there exists a configuration $(\rho_*,\ab_*)\in\cH_0$ and a sequence $(\Phi_n)$ such that $\Phi_n\to+\infty$ and
\[
\begin{gathered}
|\psi_{\Phi_n}|\to \rho_*\text{ weakly in }H^1(\Omega;\R)\text{ and strongly in }L^2(\Omega;\R),\\
\Phi_n(\Ab_{\Phi_n}-\Fb)\to \ab_*\text{ weakly in }H^2(\Omega;\R^2) \text{ and strongly in }H^1(\Omega;\R^2).
\end{gathered}\]
\end{proposition}
\begin{proof}
Let $\rho_{\Phi}=|\psi_\Phi|$ and $\ab_\Phi=\Phi(\Ab-\Fb)$. Thanks to Proposition~\ref{prop:est-min},  we have
\[\|\psi_\Phi\|_2^2\leq |\Omega|, \quad \|(-\ii\nabla-\Phi\Ab)\psi_\Phi\|_2^2\leq \kappa^2 |\Omega|,\quad\|\ab_\Phi\|_{H^2(\Omega)}^2\leq C^2\kappa^2 |\Omega|, \]
and by the diamagnetic inequality, we get
\[\|\rho_\Phi\|_2^2+\|\nabla\rho_\Phi\|^2_2\leq (1+\kappa^2)|\Omega|.\]
By the Banach--Alaoglu theorem,  we get a sequence $(\Phi_n)$ and a weak limit $(\rho_*,\ab_*)$  such that
\[\Phi_n\to+\infty,\quad \rho_{\Phi_n}\to\rho_*\text{ weakly in }H^1(\Omega;\R)\text{ and }\ab_{\Phi_n}\to\ab_*\text{ weakly in } H^2(\Omega;\R^2).\]
Thanks to  the compact embeddings of $H^1(\Omega;\R)$ and $H^2(\Omega;\R^2)$ in $L^2(\Omega;\R)$ and $H^1(\Omega;\R^2)$ respectively, we can arrange for the sequences to converge strongly in $L^2(\Omega;\R)$ and $H^1(\Omega;\R^2)$, respectively. The strong convergence in $H^1(\Omega;\R^2)$ yields
\[ \mathrm{div\,}\ab_*=0\text{ on }\Omega,\quad \nu\cdot\ab_*=0\text{ on }\partial\Omega,\]
hence $(\rho_*,\ab_*)\in\cH$. It remains to show that $\rho_*=0$ on $\omega$. For this, we write by \eqref{eq:curl-div} and Cauchy's inequality,
\[\|(-\ii\nabla-\Phi\Ab)\psi\|_{L^2(\omega)}^2\geq \frac12\|(-\ii\nabla-\Phi\Fb)\psi\|_{L^2(\omega)}^2-\widehat C\kappa^2|\Omega|,
 \]
where $\widehat C$ is a positive constant.
Then, we use the min-max principle and write
\[
    \|(-\ii\nabla-\Phi\Ab)\psi\|_{L^2(\omega)}^2\geq \frac12\lambda_1(\Phi\Fb,\omega)\|\psi\|_{L^2(\omega)}^2-\widehat C\kappa^2|\Omega|,
\]
where $\lambda_1(\Phi\Fb,\omega)$ is the lowest eigenvalue for $(-\ii\nabla-\Phi\Fb)^2$ in $L^2(\omega)$, with Neumann boundary condition on $\omega$. Knowing that $\lambda_1(\Phi\Fb,\omega)\to+\infty$ as $\Phi\to+\infty$ (see \cite[Thm.~8.1.1]{FH-b}), we deduce that
$\|\rho_{\Phi_n}\|_{L^2(\omega)}\to0$,
hence $\rho_*=0$ on $\omega$.
\end{proof}
Our next task is to refine the convergence of the minimizing order parameter when restricted to $\Omega_0$.
\begin{proposition}\label{prop:conv-min}
Let  $(\psi_\Phi,\Ab_\Phi)_{\Phi>0}\subset\cH$ be a family of configurations such that $(\psi_\Phi,\Ab_\Phi)$ is a minimizer of $\Ef$ for every $\Phi>0$. Then, there exists a sequence $(\Phi_n)$ and a function $u_*\in H^1(\Omega;\C)$ such that $\Phi_n\to+\infty$ and
\[\begin{gathered}
\cE_{\Phi_n}(\psi_{\Phi_n},\Ab_{\Phi_n})\geq \rG(\Phi_n)+o(1),\\
 U_{\lfloor\Phi_n\rfloor}^{-1}\psi_{\Phi_n}\to u_*\text{ in }H^1(\Omega_0;\C),\\
 u_*=0\text{ on }\omega,
 \end{gathered}\]
 where, for $k\in\N$, $U_k$ is the function introduced in \eqref{eq:def-U-1}.
\end{proposition}
\begin{proof}
The proof relies on a compactness argument and the extraction of a subsequence. To simplify the exposition, we will skip the reference to the subsequence.

For $\Phi>0$, we write $\Phi=\Phi_0+n$ where $n=\lfloor \Phi\rfloor$ and $\Phi_0=\{\Phi\}\in [0,1)$. Without loss of generality, we will assume that $\Phi_0$ is a fixed constant. Starting from $(\psi,\Ab)=(\psi_\Phi,\Ab_\Phi)$, we put
\[ \rho=|\psi|,\quad u=u_\Phi:=U_n^{-1}\psi,\quad \ab=\ab_\Phi:=\Phi_0\Fb+\Phi(\Ab-\Fb), \]
and we observe by \eqref{eq:curl-div} that
\[ \|\ab\|_{H^2(\Omega)}\leq C_0,\]
where $C_0$ is a positive constant, independent of $\Phi$.  Moreover, note that $u$ is defined on $\Omega_0$,  $|u|=\rho|_{\Omega_0}$, and the  following identities hold in $\Omega_0$,
\[\begin{gathered}
(-\ii\nabla-\ab)u=U_n^{-1}(-\ii\nabla-\Phi\Ab)\psi,\\
(-\ii\nabla-\ab)^2u=U_n^{-1}(-\ii\nabla-\Phi\Ab)^2\psi=\kappa^2(1-|u|^2)u.
\end{gathered}\]
Consequently, with
\begin{equation}\label{eq:def-Om-ep}
0<\varepsilon<\varepsilon_0:=\mathrm{dist}(\partial\omega,\partial\Omega),\quad \Omega_\varepsilon:=\{x\in\Omega_0\colon \mathrm{dist}(x,\partial\omega)>\varepsilon\},
\end{equation} we have  by Proposition~\ref{prop:est-min} and elliptic estimates,
\[\|u\|_{H^2(\Omega_\varepsilon)}\leq C_\varepsilon, \]
where $C_\varepsilon$ depends on $\varepsilon$ but is independent of $\Phi$.

Cantor's diagonal argument yields a sequence  and a  function $u_*\in H^2_{\rm loc}(\Omega_0;\C)$ such that,  for $0<\varepsilon<\varepsilon_0$,
\begin{equation}\label{eq:proof-conv}
 u\to u_*\text{ weakly in }H^2(\Omega_\varepsilon;\C) \mbox{ and strongly in }H^1(\Omega_\varepsilon;\C),\end{equation}
and by Propositions~\ref{prop:est-min} and ~\ref{prop:conv-density},
\begin{equation}\label{eq:proof-conv*}
\begin{gathered}
\|u_*\|_\infty\leq 1,\\
\rho\to \rho_*\text{ in }L^2(\Omega;\R)\text{ with }\rho_*|_{\omega}=0,\\
\ab\to \Phi_0\Fb+\ab_*\text{ in }H^1(\Omega;\R^2).
\end{gathered}
\end{equation}
Consequently,  $|u_*|=\rho_*|_{\Omega_0}$.  Moreover, by monotone convergence
\[\int_{\Omega_0}|(-\ii\nabla-\Phi_0\Ab_*)u_*|^2\dd x=\lim_{\varepsilon\to0_+}\int_{\Omega_\varepsilon}|(-\ii\nabla-\Phi_0\Ab_*)u_*|^2\dd x, \]
where
\[\Ab_*=\begin{cases}\Fb+\Phi_0^{-1}\ab_*&\text{if }\Phi_0>0,\\
\Fb&\text{if }\Phi_0=0,\end{cases}\]
and we have by \eqref{eq:proof-conv}-\eqref{eq:proof-conv*} and Proposition~\ref{prop:est-min},
\begin{multline*}\forall\varepsilon\in(0,\varepsilon_0),\quad \int_{\Omega_\varepsilon}|(-\ii\nabla-\Phi_0\Ab_*)u_*|^2\dd x\\
\leq \limsup_{\Phi\to+\infty}\int_{\Omega}|(-\ii\nabla-\Phi\Ab)\psi|^2\dd x\leq \kappa^2|\Omega|,\end{multline*}
hence it follows that
\[\int_{\Omega_0}|(-\ii\nabla-\Phi_0\Ab_*)u_*|^2\dd x \leq \kappa^2|\Omega|.\]
 Thus, we have $(u_*,\Ab_*)\in\cH_0$ and, for $0<\varepsilon<\varepsilon_0$,
\[\begin{aligned}
\En&=\Ef(\psi,\Ab)\\
&\geq \int_{\Omega_\varepsilon}\Bigl(|(-\ii\nabla-\ab)u|^2+\frac{\kappa^2}2|u|^4\Bigr)\dd x-\kappa^2\int_{\Omega}\rho^2\dd x+\int_\Omega|\curl\ab|^2\dd x .
\end{aligned}\]
Sending $\Phi\to+\infty$ (along the sequence $(\Phi_n)$) and then $\varepsilon\to0$, we get by monotone convergence
\[\liminf_{\Phi_n\to+\infty}\En\geq \cG_{\Phi_0}(u_*,\Ab_*)\geq \rG(\Phi_0).\]
To finish the proof, we use Propositions~\ref{prop:ub} and \ref{prop:GL-eff}, and this yields that $(u_*,\Ab_*)$ is a minimizer of $\cG_{\Phi_0}$.
\end{proof}

Collecting Proposition~\ref{prop:ub} and \ref{prop:conv-min}, we get the conclusion of Theorem~\ref{thm:GL}. Corollary~\ref{cor:oscillations1} also results from a standard argument, and we refer to the proof of \cite[Corollary~1.5]{KP} for details.

\section{The effective eigenvalue and oscillations under local fields}
\label{sec:eigenvalueapproximation}

In this section we prove Theorem~\ref{thm:ev}, from which Corollary~\ref{cor:oscillations2} follows by a standard argument (see \cite[Proof of Thm.~1.7(2)]{KP}).

For every $\Phi>0$, let $u_\Phi$ denote a normalized ground state of $\lambda_1^\Omega(\Phi)$.
Similar to Proposition~\ref{prop:conv-density}, we can prove that the function $\rho_\Phi:=|u_\Phi|$ is bounded in $H^1(\Omega;\R)$.
\begin{proposition}\label{prop:con-gs}
There exist positive  constants $C_1,\Phi_1$ such that, for all $\Phi\geq \Phi_1$, we have
\[\|\rho_\Phi\|_{H^1(\Omega;\R)}\leq C_1\]
and
\[\int_{\omega}|\rho_\Phi|^2\dd x\leq \frac{C_1}{\Phi}.\]
\end{proposition}
\begin{proof}
The first inequality is the consequence of Proposition~\ref{prop:Hf} and the non-asymptotic bound in \eqref{eq:comparison}. The second inequality follows from the min-max principle,
\[ \lambda_1(\Phi\Fb,\omega)\int_\omega|\rho_\Phi|^2\dd x\leq \lambda_1^\omega(\Phi), \]
where $\lambda_1(\Phi\Fb,{\Bl \omega})$ is the lowest eigenvalue for $(-\ii\nabla-\Phi\Fb)^2$ in $L^2(\omega)$, with Neumann boundary condition on $\omega$. We obtain the desired inequality thanks to the known asymptotics  $\lambda_1(\Phi\Fb,\omega)\sim\Theta_0 |\curl\Fb|\Phi$, where $\Theta_0\approx 0.59$ is the de\,Gennes constant (see \cite[Thm.~8.1.1]{FH-b}).
\end{proof}

Next we define the function  $v_\Phi$ on $\Omega_0$ as
\[v_\Phi=U_n^{-1}u_{\Phi},\quad n=\lfloor \Phi\rfloor,\]
where $U_n$ is the function introduced in \eqref{eq:def-U-1}. Similar to Proposition~\ref{prop:conv-min}, we can prove that $v_\Phi$ is bounded locally in $H^2(\Omega_0)$.

\begin{proposition}\label{prop:con-gs*}
Let $0<\varepsilon<\varepsilon_0:=\mathrm{dist}(\partial\omega,\partial\Omega)$.  There exist positive  constants $C_\varepsilon,\Phi_\varepsilon$ such that, for all $\Phi\geq \Phi_\varepsilon$, we have
\[\|v_\Phi\|_{H^2(\Omega_\varepsilon;\C)}\leq C_\varepsilon,\]
where $\Omega_\varepsilon$ is introduced in \eqref{eq:def-Om-ep}.
\end{proposition}
\begin{proof}
With $\Phi_0=\{\Phi\}$, we have
\[
(-\ii\nabla-\Phi_0\Fb)^2v_\Phi=\lambda_1^\Omega(\Phi)v_\Phi\text{ on }\Omega_0.
\]
The bound in $H^2(\Omega_\varepsilon,\C)$ is the result of elliptic $L^2$ estimates.
\end{proof}
\begin{remark}\label{rem:vb}
Thanks to Proposition~\ref{prop:con-gs}, the function $v_\Phi$ is almost normalized in $L^2(\Omega_0)$; in fact, we have as $\Phi\to+\infty$,
\[\int_{\Omega_0}|v_\Phi|^2\dd x=1+\mathcal O(1/\Phi). \]
\end{remark}

\begin{proof}[Proof of Theorem~\ref{thm:ev}]
Thanks to \eqref{eq:comparison}, we have
\[ \limsup_{\Phi\to+\infty} \bigl(\lambda_1^\Omega(\Phi)-\lambda_1^{\Omega_0}(\Phi)\bigr)\leq 0,\]
so it suffices to show that
\[ \liminf_{\Phi\to+\infty} \bigl(\lambda_1^\Omega(\Phi)-\lambda_1^{\Omega_0}(\Phi)\bigr)\geq 0.\]
If this does not hold, then there exist a positive constant $c_*$ and an unbounded set $\cI\subset\R_+$ such that,
\begin{equation}\label{eq:hyp}
\forall\, \Phi\in \cI, \quad \lambda_1^\Omega(\Phi)-\lambda_1^{\Omega_0}(\Phi)<-c_*.
\end{equation}
Since $(\rho_\Phi)_{\Phi\in \cI}$ is bounded in $H^1(\Omega;\R)$, and $(v_\Phi)_{\Phi\in\cI}$ is bounded in $H^2(\Omega_\varepsilon;\C)$ for every $\varepsilon<\varepsilon_0$, there are functions
\[\rho_*\in H^1(\Omega;\R),\quad v_*\in H^2_{\rm loc}(\Omega_0;\C),\]
and a sequence $(\Phi_k)\subset\cI$ such that
\[ \Phi_k\to+\infty,\quad \{\Phi_k\}\to \Phi_*\in[0,1],\]
and
\[\rho_{\Phi_k}\to\rho_*\text{ strongly in }L^2(\Omega;\R),\quad v_{\Phi_k}\to v_*\text{ strongly in }H^1(\Omega_\varepsilon).\]
Moreover,  we have
\[ \int_{\Omega_\varepsilon}|(-\ii\nabla-\{\Phi_k\}\Fb)v_{\Phi_k}|^2\dd x=\int_{\Omega_\varepsilon}|(-\ii\nabla-\Phi_k\Fb)u_{\Phi_k}|^2\dd x\leq \lambda_1^\Omega(\Phi_k)\leq \lambda_1^{\Omega_0}(1/2),\]
so that, by taking $k\to+\infty$ then $\varepsilon\to0_+$ we have by monotone convergence
\[ \int_{\Omega_0}|\nabla v_*|^2\dd x\leq \|\Fb\|_\infty^2+ \lambda_1^\Omega(1/2).\]
Thus, $v_*\in H^1(\Omega_0;\C)$ and $|v_*|=\rho_*$. By Proposition~\ref{prop:con-gs} we get that $\rho_*|_{\partial\omega}=0$, hence $v_*|_{\partial\omega}=0$, and  by Remark~\ref{rem:vb}, we have
$\int_{\Omega_0}|v_*|^2\dd x=1$.

Finally, for every $0<\varepsilon<\varepsilon_0$, we have
\[\lambda_1^{\Omega}(\Phi_k)\geq  \int_{\Omega_\varepsilon}|(-\ii\nabla-\{\Phi_k\}\Fb)v_{\Phi_k}|^2\dd x,\]
and by taking $k\to+\infty$, we get 
\[\limsup_{k\to+\infty}\lambda_1^{\Omega}(\Phi_k)\geq
\int_{\Omega_\varepsilon}|(-\ii\nabla-\Phi_*\Fb)v_{*}|^2\dd x,\]
then by taking $\varepsilon\to0_+$, we get by monotone convergence and the min-max principle
\[\limsup_{k\to+\infty}\lambda_1^{\Omega}(\Phi_k)\geq
\int_{\Omega_0}|(-\ii\nabla-\Phi_*\Fb)v_{*}|^2\dd x\geq \lambda_1^{\Omega_0}(\Phi_*),\]
and eventually
\[ \limsup_{k\to+\infty}\bigl(\lambda_1^{\Omega}(\Phi_k)-\lambda_1^{\Omega_0}(\Phi_k)\bigr)\geq 0\]
which violates \eqref{eq:hyp}.
\end{proof}

\subsection*{Acknowledgments}
{\small This work started when AK visited Lund University in the period  Sep 2022--28 Feb 2023, and the authors would like to thank the Knut and Alice foundation for the financial support (grant KAW 2021.0259). AK is partially supported by CUHKSZ, grant no.\ UDF01003322. Part of this work was carried out when MS visited the Chinese University of Hong Kong (Shenzhen) during March 2024.}

%\begin{thebibliography}{99}
%
%\bibitem{1} B. D. O. Anderson and J. B. Moore, {\it Linear optimal
%control},  Prentice-Hall,  Englewood Cliffs, NJ, 1971.
%
%
%\bibitem{2}S. Aubry and P.Y. Le Daeron, {\it The discrete
%Frenkel-Kontorova model and its extensions I},  Physica D  {\bf 8} (1983), 381--422.
%
%\bibitem{3} J. Baumeister, A. Leitao and G. N. Silva, {\it On
%the value function for nonautonomous optimal control problem with
%infinite horizon},  Systems Control Lett. {\bf 56}  (2007),  188--196.
%
%
%\bibitem{4} J. Blot, {\it Infinite-horizon Pontryagin
%principles without invertibility},  J. Nonlinear Convex Anal.
%{\bf 10} (2009),  177--189.
%
%\bibitem{5} J. Blot and P. Cartigny, {\it  Optimality in
%infinite-horizon variational problems under sign conditions},
%J. Optim. Theory Appl. {\bf 106} (2000),  411--419.
%
%
%
%
%
%
%
%
%\end{thebibliography}

\end{document}